\documentclass[11pt,table,reqno]{amsart}

\usepackage{amsfonts,amssymb,graphicx,color,tikz,amsmath,amsthm}

\tikzset{vertex/.style={circle,draw,fill,inner sep=0pt,minimum size=1mm}}

\usepackage{todonotes}
\usepackage{hyperref}
\usepackage[all]{xy}

\newtheorem{thm}{Theorem}
\newtheorem{cor}[thm]{Corollary}
\newtheorem{lem}[thm]{Lemma}
\newtheorem{prop}[thm]{Proposition}

\newtheorem{quest}[thm]{Question}

\theoremstyle{definition}
\newtheorem{definition}[thm]{Definition}
\newtheorem{exa}[thm]{Example}

\newtheorem{rmk}[thm]{Remark}

\numberwithin{thm}{section}

\newcommand{\Z}{\mathbb{Z}}
\newcommand{\R}{\mathbb{R}}

\usepackage{stmaryrd} 
\newcommand{\adj}{\sim}

\DeclareMathOperator{\id}{id}
\DeclareMathOperator{\NP}{NP}

\begin{document}

\baselineskip= 20.80pt

\title{On digital H-spaces}

\thanks{
The second author was supported by the National Research Foundation of Korea (NRF) grant funded by the Korean government (MSIT) (No. 2020R1A2B6004407)}

\author{Wayne A. Johnson, Dae-Woong Lee, and P. Christopher Staecker}

\date{\today}

\address{
Department of Mathematics, Truman State University, 100 E Normal Ave,
Kirksville, Missouri, 63501
}
\email{wjohnson@truman.edu}

\address{
Department of Mathematics, and Institute of Pure and Applied Mathematics, Jeonbuk National University,
567 Baekje-daero, Deokjin-gu, Jeonju-si, Jeollabuk-do 54896, Republic of Korea
}
\email{dwlee@jbnu.ac.kr}

\address{
Mathematics Department, Fairfield University, 1703 North Benson Rd, Fairfield, CT 06824-5195
}
\email{cstaecker@fairfield.edu}


\begin{abstract}
In this article, we investigate properties of digital H-spaces in the graph theoretic model of digital topology. As in prior work, the results obtained often depend fundamentally on the choice between $\NP_1$ and $\NP_2$ product adjacencies. We explore algebraic properties of digital H-spaces preserved under digital homotopy equivalence, and we give a general construction that produces examples of digital H-spaces which are not homotopy-equivalent to digital topological groups in both categories. Further, we show that this construction essentially classifies all $\NP_2$-digital H-spaces. In a short appendix, we resolve a question that was left unresolved in \cite{dtg}, and complete the full classification of digital topological groups.
\end{abstract}

\maketitle

\section{Introduction}
In the late nineteenth century, Henri Poincar\'e considered the first homotopy group and the notion of homology which are the fundamental methods to classify and distinguish many kinds of topological spaces up to classical homotopy relations. Classical homotopy theory, as one of the main topics in algebraic topology, has been developed in the homotopy category of pointed CW-spaces and homotopy classes of base point preserving continuous functions between CW-spaces. General approaches in classical homotopy theory have been to change algebraic, analytic, or geometric problems into homotopy problems by using various algebraic tools.

An H-group is a homotopy-theoretic generalization of the concept of a group and is an Eckmann-Hilton dual of a co-H-group; see \cite{A1}. An H-space lacks the axioms of homotopy-associativity and homotopy-commutativity from an H-group. 

Digital homotopy theory is a relatively new branch of mathematics and computer science that adapts concepts from classical homotopy theory to the setting of digital images or digital spaces. It seeks to define and study invariants that can classify digital images up to digital homotopy equivalence.

The graph-theoretical approach to digital topology was introduced by A. Rosenfeld in the 1970s; see \cite{R}. Work in this area has continued from the 1980s to the present day by many authors. This digital theory has developed in parallel with $A$-homotopy theory \cite{BL} and $\times$-homotopy theory \cite{Do} in the context of abstract graph theory. Those two theories have defined invariants and concepts which in many cases match ideas independently developed for the topology of digital images. All of these theories depend strongly on the details of choices made when constructing products. In the digital topology literature, the two canonical products are denoted $\NP_1$ and $\NP_2$. Generally speaking, $A$-theory works in categories exclusively using $\NP_1$ products, while the $\times$-homotopy theory uses $\NP_2$ products.

Some early work on digital H-spaces appears in \cite{EK} and \cite{EK1}, which make basic definitions and prove some standard results in the $\NP_2$ category (sometimes using pointed homotopies). No examples are given of H-spaces which are not contractible. 
The actions on the set of pointed digital homotopy (associative or commutative) operations to create new digital homotopy (resp., associative or commutative) operations were developed in \cite{D2} in the $\NP_2$ category, and the near-ring structure on the set of all pointed digital homotopy classes of digital Hopf functions between pointed digital Hopf groups was investigated in \cite{D1}.

This paper continues work in \cite{dtg}, which defined and studied \emph{digital topological groups} in both the $\NP_1$ and $\NP_2$ categories. In this paper we define both $\NP_1$ and $\NP_2$-digital H-spaces by relaxing the associativity and invertibility conditions.

This paper is organized as follows. In Section \ref{prem}, we present well-known background material on digital images, digitally continuous functions, and normal product adjacency relations, and we introduce $\NP_i$-digital H-spaces. In Section \ref{prop}, we explore algebraic properties of digital H-spaces that are preserved by homotopy equivalence. In Section \ref{irr}, we explore properties of irreducible digital H-spaces in both categories. In Section \ref{cont}, we show that every connected $\NP_2$-digital H-space is contractible, and thus homotopy equivalent to the trivial digital topological group. In Section \ref{pointed}, we provide an example of a digital H-space that is homotopic to a digital topological group but not pointed homotopic to a group. In Section \ref{examples}, we present a general construction that produces examples of $\NP_1$ and $\NP_2$-digital H-spaces that are not homotopic to digital topological groups and fully classify the $\NP_2$-digital H-spaces. Finally, in a brief appendix, we resolve a question that was left unresolved in \cite{dtg}, and complete the full classification of $\NP_i$-digital topological groups for $i\in\{1,2\}$.

\bigskip

\section{Preliminaries} \label{prem}

A \emph{digital image} $(X, \kappa)$ consists of a finite set $X$ of points in $\Z^n$ with some \emph{adjacency relation} $\kappa$ which is reflexive and symmetric. 
Typically the adjacency relation is based on some notion of nearness of points in $\Z^n$. This style of digital topology has its origins in the work of Rosenfeld and others; see \cite{R} for an early work. We will make use of the notation $x \adj_\kappa y$ when $x$ is adjacent to $y$ by the adjacency relation $\kappa$. The particular adjacency relation will usually be clear from context, and in this case we will omit the subscript.

For subsets of $\Z$, the standard adjacency to use is given by $a\sim b$ if and only if $|a-b|\le 1$. When $n\geq 2$ there is no canonical ``standard adjacency'' to use in $\Z^n$ that corresponds naturally to the standard topology of $\R^n$. In the case of $\Z^2$, for example, at least two different adjacency relations are natural: we can view $\Z^2$ as a rectangular lattice connected by the coordinate grid, so that each point is adjacent to 4 neighbors; or we can additionally allow diagonal adjacencies so that each point is adjacent to 8 neighbors. Generalizations of these choices of allowable diagonals gives the $c_u$-adjacencies on $\Z^n$, see \cite{LB3}.

Any digital image $(X,\kappa)$ can naturally be viewed as a finite reflexive graph with vertex set $X$ and an edge connecting $x,y\in X$ if and only if $x\adj y$. Conversely, any finite graph may be considered as a digital image in some $\Z^n$ with $c_n$-adjacency (this is the adjacency which allows all available diagonals):

\begin{thm}[\cite{los23}, Proposition 2.5]\label{embeddingthm}
Let $X$ be a finite simple graph of $k$ vertices. Then, for $n = k-1$, $X$ may be embedded as a digital image $X \subset [-1,1]_\Z^n$ with $c_n$-adjacency.
\end{thm}

(Above, by ``embedded'' we mean that $X$, considered as an abstract graph, is isomorphic as a graph to the digital image $X\subset  [-1,1]_\Z^n$ with $c_n$-adjacency.)

We will, whenever convenient, describe a digital image $(X,\kappa)$ in graph-theoretic terms. Specifically, our definitions above of paths, connectedness, and components correspond exactly to the same concepts in graph theory. All specific examples of digital images given in later sections will be presented simply as abstract graphs, not specifically embedded in some $\Z^n$.

A function $f : (X,\kappa) \rightarrow (Y,\lambda)$ between digital images is called a {\it $(\kappa,\lambda)$-continuous function} \cite {LB1} when for any $x_1,x_2\in X$, if $x_1 \adj x_2$, then  $f(x_1) \sim f(x_2)$.
When the adjacencies are clear, we simply say that $f$ is \emph{continuous}.
Any composition of digitally continuous functions is digitally continuous.

Let $a$ and $b$ be nonnegative integers with $a < b$. A {\it digital interval} \cite {LB0} is a finite set of the form
$$
[a,b]_{\mathbb{Z}} = \{ z \in \mathbb Z ~\vert~ a \leq z \leq b \},
$$
with the usual adjacency relation in $\mathbb{Z}$.

A \emph{path} from $x$ to $y$ in a digital image $X$ is a digitally continuous function
$$
f : [0, m]_{\mathbb{Z}} \rightarrow X.
$$
with $f(0)=x$ and $f(m)=y$. Equivalently, a path from $x$ to $y$ can be specified by a sequence of adjacent points:
\[ x = x_0 \sim x_1 \sim \dots \sim x_m = y. \]

A digital image $X$ is said to be \emph{connected} \cite {R} if for every pair $x,y$ of  points of $X$, there exists a path from $x$ to $y$. Given some $x\in X$, the set of all points of $X$ having a path to $x$ is the \emph{connected component} of $x$. A digital image is digitally connected if and only if it consists of a single component.

A $(\kappa,\lambda)$-continuous function of digital images
$$
f : (X,\kappa) \rightarrow (Y,\lambda)
$$
is called a {\it $(\kappa,\lambda)$-isomorphism} if $f$ is a bijection set-theoretically, and its inverse $f^{-1} : (Y,\lambda) \rightarrow (X,\kappa)$ is $(\lambda,\kappa)$-continuous. In this case, $(X,\kappa)$ and $(Y,\lambda)$ are said to be {\it $(\kappa,\lambda)$-isomorphic}; see  \cite {LB0} and \cite{LB2}. We often make use of the following fact: if $f$ is an isomorphism, then $x\adj y$ if and only if $f(x)\adj f(y)$.

Given two digital images $X_1$ and $X_2$, we can consider the Cartesian product $X_1\times X_2$ as a digital image, but there are several natural choices for the adjacency relations to be used in the Cartesian product, analogous to the various $c_u$-adjacencies. The most natural product adjacencies are the \emph{normal product adjacencies}, which were defined by Boxer as follows:

\begin{definition}(\cite{LB4})
Let $\{(X_i,\kappa_i) \mid i = 1,2, \cdots, n\}$ be an indexed family of digital images.
Then for some $u \in \{1,2, \dots,n\}$, the \emph{normal product adjacency} $\NP_u(\kappa_1,\kappa_2,\dots,\kappa_n)$ is the adjacency relation on $\prod_{i=1}^n X_i$ defined by: $(x_1,x_2,\dots, x_n)$ and $(x'_1,x'_2, \dots, x'_n)$ are adjacent if and only if their coordinates are adjacent in at most $u$ positions, and equal in all other positions.
\end{definition}

In this paper, our products generally have the form $X\times X$ for some digital image $(X,\kappa)$. On such a Cartesian product, the two natural adjacencies to choose from are $\NP_1(\kappa,\kappa)$ and $\NP_2(\kappa,\kappa)$. When unambiguous, we will abbreviate $\NP_i(\kappa,\kappa)$ as simply $\NP_i$ for $i\in \{1,2\}$. Clearly, if two points $(x_1, x_2)$ and $(y_1, y_2)$ in $X\times X$ are $\NP_1$-adjacent, then they are also $\NP_2$-adjacent. As we will see, the topological structure of the Cartesian product depends strongly on the choice between $\NP_1$ and $\NP_2$.

We will require two slightly different types of products of maps, where continuity of the product map depends on the choice between $\NP_1$ and $\NP_2$. The following appears as Theorem 4.1 of \cite{LB4}.
\begin{prop}\label{product}
Let $X_i$ and $Y_i$ be digital images for $i\in \{1,2\}$, and let $f_1:X_1 \to Y_1$ and $f_2:X_2 \to Y_2$ be continuous. Then the map $f_1\times f_2: X_1 \times X_2 \to Y_1 \times Y_2$ is both $(\NP_1,\NP_1)$- and $(\NP_2,\NP_2)$-continuous.
\end{prop}

Our other type of product map is formed as a product of two maps having the same domain. The result this time is weaker, only applying to $\NP_2$-adjacency.
\begin{prop}\label{np2product}
Let $X$ and $Y$ be digital images, and let $f,g:X\to Y$ be continuous. Let $\kappa$ be the adjacency on $X$. Then the map $(f,g):X\to Y\times Y$ is $(\kappa,\NP_2)$-continuous.
\end{prop}
\begin{proof}
Take $a,b\in X$ with $a\sim b$, and we must show that $$(f,g)(a) = (f(a),g(a))$$ and $$(f,g)(b) = (f(b),g(b))$$ are $\NP_2$-adjacent. But this is clear because $f(a)\sim f(b)$ and $g(a)\sim g(b)$ by continuity of $f$ and $g$.
\end{proof}

We note that Proposition \ref{np2product} is not true if $\NP_2$ is replaced by $\NP_1$, even in very simple cases. For example if $f=g$ is the identity map $\id_X$, and $a$ and $b$ are distinct adjacent points of $X$, then we will have $(f,g)(a) = (a,a)$ and $(f,g)(b) = (b,b)$, but $(a,a) \not \sim_{\NP_1} (b,b)$ because $a\neq b$. In particular, this means that the traditional diagonal map $\Delta_X:X\to X\times X$ is not $(\kappa,\NP_1)$-continuous, so we will usually avoid it in our definitions. We will, however, often make use of the pair $(\id_X, c)$, where $c$ is a constant. This map is always $(\kappa,\NP_1)$ continuous.

\begin{definition} (\cite {LB0, LB3}) \label{homotopydef}
Let $X$ and $Y$ be digital images, and $i\in \{1,2\}$. An \emph{$\NP_i$-digital homotopy} is an $\NP_i$-continuous map
\[ H:X\times [0,m]_\Z \to Y \]
for some natural number $m$.
If $f, g:X \to Y$ are continuous functions and $H(x,0)=f(x)$ and $H(x,m)=g(x)$, then we say $H$ is an \emph{$\NP_i$-homotopy from $f$ to $g$}, and we say $f$ and $g$ are \emph{$\NP_i$-homotopic}. In this case we write $f\simeq_i g$. 
When $m=1$, we say the homotopy is \emph{single step}.
\end{definition}


Two digital images $X$ and $Y$ are said to have the {\it same $\NP_i$-homotopy type}, or to be {\it $\NP_i$-homotopy equivalent}, if there exist continuous functions $f : X\to Y$ and $g : Y\to X$ such that $f \circ g \simeq_i \id_Y$ and $g\circ f \simeq_i \id_X$. In this case, each of the functions $f$ and $g$ is said to be an {\it $\NP_i$-homotopy equivalence}, and $g$ is called an $\NP_i$-homotopy inverse of $f$ (and vice versa).

Now we are ready to define our main object of study, the digital H-space. Recall that, although a general pair of maps $(f,g)$ may not be $\NP_1$-continuous, this pair is always $\NP_i$-continuous for any $i\in \{1,2\}$ when one of the maps is a constant.

\begin{definition}\label{wdc}
For some $i\in \{1,2\}$, an \emph{$\NP_i$-digital H-space} is a triple $(X,e,\mu)$, where $X$ is a digital image with adjacency relation $\kappa$ and $e\in X$, and $\mu:X\times X \to X$ is an ($\NP_i(\kappa, \kappa),\kappa)$-continuous map such that the diagram
$$
\xymatrix@C=20mm @R=15mm{
X \ar@/_/[dr]_-{\id_X} \ar[r]^-{(\id_X, c_e)} & X \times X \ar[d]_-{\mu} &X \ar[l]_-{(c_e, \id_X)} \ar@/^/[dl]^-{\id_X}  \\
&X }
$$
commutes up to digital homotopy; that is,
$$
\mu \circ (\id_X, c_e) \simeq_i \id_X
$$ 
and 
$$
\mu\circ (c_e , \id_X) \simeq_i \id_X,
$$
where $c_e : X \to X$ is the constant function with constant value $e$.
In this case, the function $\mu : X \times X \to X$ is called an \emph{$\NP_i$-digital multiplication} on $X$.
If $$\mu \circ (\id_X, c_e) = \id_X = \mu\circ (c_e , \id_X),$$ we say $(X,e,\mu)$ is \emph{unital}.
\end{definition}

To abbreviate our notation, we say that an ($\NP_i(\kappa, \kappa),\kappa)$-continuous function $\mu:X\times X \to X$ is simply $\NP_i$-continuous.

We note that the homotopies we require in the definition above are not assumed to preserve the basepoint $e$. We include a more detailed discussion of pointed vs non-pointed homotopies in Section \ref{pointed}.

\begin{definition}
Let $(X, e, \mu)$ be an $\NP_i$-digital H-space with an $\NP_i$-digital multiplication $\mu : X \times X \rightarrow X$ for $i =1,2$. A continuous function $\alpha : X \rightarrow X$ is said to be a \emph{left homotopy-inverse}
if $\mu\circ (\alpha,\id_X)$ is $(\kappa, \kappa)$-continuous, and the diagram
$$
\xymatrix@C=25mm @R=15mm{
X \ar@/_/[dr]_-{c_e} \ar[r]^-{(\alpha, \id_X)} & X \times X \ar[d]^-{\mu}   \\
&X }
$$
commutes up to $\NP_i$-digital homotopy; that is,
\begin{equation}\label{Lee17-1}
\mu \circ (\alpha, \id_X) \simeq_i c_e.
\end{equation}
If $$\mu \circ (\alpha, \id_X) = c_e,$$ we say $\alpha$ is a \emph{left inverse}. 
Similarly, a continuous function $\beta : X \rightarrow X$ is said to be a \emph{right homotopy-inverse}
if $\mu\circ (\id_X,\beta)$ is $\NP_i$-continuous, and $$\mu \circ (\id_X,\beta) \simeq_i c_e.$$ If $$\mu \circ (\id_X,\beta) = c_e,$$ we say $\beta$ is a \emph{right inverse}.
\end{definition}

\begin{rmk}\label{htpinv}
We note that homotopy-invertibility is not a very natural property in the $\NP_1$ category, because a function $(\alpha, \id_X)$ may not in general be $\NP_1$ continuous. 

In fact, if $\alpha:X\to X$ is a left inverse, then it is automatically an isomorphism, and in this case $(\alpha,\id_X)$ will be discontinuous: for distince $x,y\in X$ with $x\sim y$, we will have $$(\alpha,\id_X)(x) = (\alpha(x),x) \not \sim (\alpha(y),y) = (\alpha,\id_X)(y).$$ Thus in the $\NP_1$ category, even an exact inverse function will not be a homotopy-inverse, and so we do not expect homotopy-invertibility to be a natural property for the $\NP_1$ category.
\end{rmk}

\begin{definition}
Let $(X, e, \mu)$ be an $\NP_i$-digital H-space with an $\NP_i$-digital multiplication $\mu : X \times X \rightarrow X$. This H-space is said to be  \emph{homotopy-associative} if the following diagram
\begin{equation}\label{Lee18}
\xymatrix@C=25mm @R=15mm{
X \times X \times X \ar[d]_-{\id_X \times \mu} \ar[r]^-{\mu \times \id_X} & X \times X \ar[d]^-{\mu}  \\
X \times X \ar[r]^-{\mu} &X }
\end{equation}
commutes up to $\NP_i$-digital homotopy; that is,
$$
\mu \circ (\id_X \times \mu) \simeq_i \mu \circ (\mu \times \id_X).
$$
If $$\mu \circ (\id_X \times \mu) = \mu \circ (\mu \times \id_X),$$ we say $(X,e,\mu)$ is \emph{associative}.
\end{definition}

Given an $\NP_i$-digital H-space $(X,e,\mu)$, for each $x\in X$, let $\mu_x:X\to X$ be given by $$\mu_x(y) = \mu(x,y).$$ Similarly define $\nu_x:X\to X$ by $$\nu_x(y) = \mu(y,x).$$ 

\begin{lem}\label{muhomotopic}
Let $(X,e,\mu)$ be an $\NP_i$-digital H-space. If $x,y\in X$ are in the same component of $X$, then $\mu_x \simeq_i \mu_y$ and $\nu_x \simeq_i \nu_y$.
\end{lem}
\begin{proof}
We will give the proof for the statement about $\mu_x$. The proof for $\nu_x$ is similar.

Since $x$ and $y$ are in the same component, there is a path along adjacencies connecting $x$ and $y$. By induction and transitivity of the homotopy relation, it suffices to show that if $x\sim y$, then $\mu_x\simeq_i \mu_y$. In fact we will show that $\mu_x$ and $\mu_y$ are homotopic in a single step.

We argue the cases $i=1$ and $i=2$ separately. For $i=1$, by Proposition 1.4 of \cite{stae21}, we must show that   $\mu_x(a)\sim \mu_y(a)$ for every $a\in X$. Since $x\sim y$ we have $$(x,a) \sim_1 (y,a),$$ and since $\mu$ is $\NP_1$-continuous we have $$\mu(x,a) \sim \mu(y,a)$$ which means $$\mu_x(a)\sim \mu_y(a)$$ as desired.

For $i=2$, by Proposition 2.4 of \cite{stae21}, we must show that if $a \sim b$, then $$\mu_x(a) \sim \mu_y(b).$$ Since both $a\sim b$ and $x\sim y$, we have $$(x,a)\sim_2 (y,b).$$ Then since $\mu$ is $\NP_2$-continuous we have $$\mu(x,a) \sim \mu(y,b)$$ which means $$\mu_x(a)\sim \mu(y,b)$$ as desired.
\end{proof}

Let $X_e\subseteq X$ be the component of $e$. Since $$\mu_e \simeq_i \mu\circ (\id_X,c_e) \simeq_i \id_X,$$ the above gives:
\begin{thm}\label{homotopictoid}
Let $(X,e,\mu)$ be an $\NP_i$-digital H-space. Then if $x\in X_e$, we have $$\mu_x \simeq_i \id_X \simeq_i \nu_x.$$
\end{thm}

%
%
%

\bigskip

\section{Homotopy equivalence of H-spaces}  \label{prop}
In this section, we examine which algebraic properties will be preserved by homotopy equivalence of digital H-spaces. We will see that homotopy-associativity is preserved, homotopy-invertibility is preserved in the $\NP_2$ category, and the existence of a unit may not be preserved.

We will use a special form of homotopy equivalence which preserves the H-space structure:
\begin{definition}\label{hspaceequiv}
Let $(X,e_X,\mu_X)$ and $(Y,e_Y,\mu_Y)$ be two $\NP_i$-digital H-spaces. We say they are \emph{H-equivalent} if there are continuous pointed maps $f:(X,e_X) \to (Y,e_Y)$ and $g:(Y,e_Y) \to (X,e_X)$ such that
\begin{itemize}
    \item $f \circ g \simeq_i \id_Y$;
    \item $g\circ f \simeq_i \id_X$;
    \item $f \circ \mu_X \simeq_i \mu_Y \circ (f\times f)$ and
    \item $g \circ \mu_Y \simeq_i \mu_X \circ (g\times g)$.
\end{itemize}
\end{definition}

It is easy to verify that $H$-equivalence is an equivalence relation.

If we have a homotopy equivalence of digital images $X$ and $Y$, with an H-space structure on $X$, this can be used to construct an H-space structure on $Y$ as in the following Theorem. A version of this result also appears in Theorem 3.9 of \cite{EK1}.
\begin{thm}\label{inducedmult}
Let $(X,e_X,\mu_X)$ be an $\NP_i$-digital H-space, and let $(Y,e_Y)$ be a pointed digital image such that $X$ and $Y$ are $\NP_i$-homotopy equivalent by pointed maps $f:(X,e_X) \to (Y,e_Y)$ and $g:(Y,e_Y) \to (X,e_X)$. Let $\mu_Y:Y\times Y \to Y$ be defined by
\[\mu_Y = f \circ \mu_X \circ (g\times g). \]
Then $(Y,e_Y,\mu_Y)$ is an $\NP_i$-digital H-space, and $(X,e_X,\mu_X)$ and $(Y,e_Y,\mu_Y)$ are H-equivalent.
\end{thm}
\begin{proof}
To show that $(Y,e_Y,\mu_Y)$ is an $\NP_i$-digital H-space, we have:
\[ \begin{split}
\mu_Y \circ (\id_Y, c_{e_Y}) &= f \circ \mu_X \circ (g\times g) \circ (\id_Y , c_{e_Y})  \\
&= f \circ \mu_X \circ (\id_X , c_{e_X}) \circ g, \\
&\simeq_i f \circ \id_X \circ g \\
&\simeq_i \id_Y. 
\end{split}
\]
Similarly, we can show that $$\mu_Y \circ (c_{e_Y}, \id_Y) \simeq_i \id_Y,$$ and so $(Y,e_Y,\mu_Y)$ is an $\NP_i$ digital H-space.

To show that they are H-equivalent, we must show that $$f\circ \mu_X \simeq_i \mu_Y \circ (f\times f)$$ and $$g\circ \mu_Y \simeq_i \mu_X \circ (g\times g).$$ We have:
\[ 
\begin{split}
\mu_Y \circ (f\times f) &= f \circ \mu_X \circ (g\times g) \circ (f\times f) \\
&= f \circ \mu_X \circ ((g\circ f) \times (g\circ f)) \\
&\simeq_i f\circ \mu_X 
\end{split}
\]
and 
\[
\begin{split}
g\circ \mu_Y &= g \circ f \circ \mu_X \circ (g\times g) \\
&\simeq_i \mu_X \circ (g\times g) 
\end{split}
\]
as desired.
\end{proof}

We now demonstrate that an H-equivalence preserves homotopy-associativity. 

\begin{thm}\label{htpeqassociative}
Let $(X,e_X,\mu_X)$ and $(Y,e_Y,\mu_Y)$ be $\NP_i$-digital H-spaces which are H-equivalent. If $(X,e_X,\mu_X)$ is homotopy-associative, then $(Y,e_Y,\mu_Y)$ is homotopy-associative.
\end{thm}
\begin{proof}

Assume that $(X,e_X,\mu_X)$ is homotopy-associative, and we will show that $(Y,e_Y,\mu_Y)$ is homotopy-associative. 

Let $f:(X,e_X) \to (Y,e_Y)$ and $g:(Y,e_Y) \to (X,e_X)$ be the maps realizing the H-equivalence, so that $$\mu_Y\circ (f\times f) \simeq_i f\circ \mu_X$$ and $$\mu_X \circ (g\times g) \simeq_i g \circ \mu_Y.$$ We have:
\[ \begin{split}
\mu_Y\circ (\mu_Y\times \id_Y) &\simeq_i f \circ g \circ \mu_Y \circ ((f\circ g \circ \mu_Y) \times \id_Y) \\
&= f \circ \mu_X \circ (g\times g) \circ ((f \circ \mu_X \circ (g\times g)) \times \id_Y) \\
&= f \circ \mu_X \circ  ((g \circ f \circ \mu_X \circ (g\times g))\times g) \\
&\simeq_i f \circ \mu_X \circ ((\mu_X \circ (g\times g))\times g) \\
&= f \circ \mu_X \circ (\mu_X \circ \id_X) \circ (g\times g\times g) \\
&\simeq_i f \circ \mu_X \circ (\id_X \times \mu_X) \circ (g\times g \times g) \\
&= f \circ \mu_X \circ (g \times (\mu_X \circ (g\times g))) \\
&\simeq_i f \circ \mu_X \circ (g\times (g\circ f \circ \mu_X \circ (g\times g))) \\
&\simeq_i \mu_Y \circ (f\times f) \circ (g \times (g\circ \mu_Y \circ (f\times f) \circ (g\times g))) \\
&= \mu_Y \circ ((f\circ g) \times (f\circ g \circ \mu_Y \circ ((f\circ g)\times (f\circ g))) \\
&= \mu_Y \circ (\id_Y \times \mu_Y)
\end{split} \]
as desired.
\end{proof}


The theorem above showed that if $X$ and $Y$ are homotopy equivalent H-spaces and $X$ is homotopy-associative, then $Y$ is homotopy-associative. The following theorem is a similar result for homotopy-inverses, but we are only able to prove it in the $\NP_2$ category. A similar result in the $\NP_1$ category may not be very meaningful in light of Remark \ref{htpinv}. 
\begin{lem}\label{htpeqinverses}
Let $(X,e_X,\mu_X)$ and $(Y,e_Y,\mu_Y)$ be H-equivalent $\NP_2$-digital H-spaces. If $(X,e,\mu)$ has a left homotopy-inverse, then $(Y,e_Y,\mu_Y)$ has a left homotopy-inverse. Similarly, if $(X,e,\mu)$ has a right homotopy-inverse, then $(Y,e_Y,\mu_Y)$ has a right homotopy-inverse. If the left and right homotopy-inverses of $X$ are homotopic, then the left and right homotopy-inverses of $Y$ are homotopic.
\end{lem}
\begin{proof}
We will prove the statement about left homotopy-inverses. The statement for right homotopy-inverses is similar.

Let $f:(X,e_X) \to (Y,e_Y)$ and $g:(Y,e_Y) \to (X,e_X)$ be the maps realizing the homotopy equivalence, and let $\alpha:X\to X$ be a left homotopy-inverse for $X$. Define $\tau:Y\to Y$ by $$\tau = f \circ \alpha \circ g,$$ and we will show that $\tau$ is a left homotopy-inverse for $Y$. First note that $\mu_Y \circ (\tau ,\id_Y)$ is continuous in the $\NP_2$ category. And for $i\in \{1,2\}$ we have:
\[ \begin{split}
\mu_Y \circ (\tau ,\id_Y)
&\simeq_i f \circ g \circ \mu_Y \circ (\tau ,\id_Y) \\
&\simeq_i f \circ \mu_X \circ (g\times g) \circ ((f\circ \alpha \circ g) , \id_Y)  \\
&= f \circ \mu_X \circ (g\circ f \circ \alpha , \id_X) \circ g \\
&\simeq_i f \circ \mu_X \circ (\alpha , \id_X)  \circ g \\
&\simeq_i f\circ c_{e_X} \circ g = c_{e_Y} \circ f \circ g \simeq_i c_{e_Y}
\end{split}
\]
and so $\tau$ is a left homotopy-inverse for $Y$.
\end{proof}

We note that the homotopies used in the proof above will exist and are continuous even in the $\NP_1$ category. But the proof as a whole may fail in the $\NP_1$ category because $\mu_Y\circ (\tau , \id_Y)$ may not be continuous.

It is natural to ask if $X$ and $Y$ are H-equivalent and $X$ is unital, then must $Y$ be unital? This is not the case, as the following example demonstrates.
\begin{exa}
Let $X = \{x_0,x_1\}$ be a digital image of 2 adjacent points, with multiplication defined by $\mu_X(x_i,x_j) = x_0$ for all $i,j$. It is easy to check that $(X,x_0,\mu_X)$ is an $\NP_2$-digital H-space which is not unital, since for example there is no $e\in X$ with $\mu_X(x_1, e) = x_1$. Let $Y = \{y_0\}$ be a digital image of 1 point, with the obvious multiplication $\mu_Y$. Then $Y$ is a unital $\NP_2$-digital H-space. 

Both $X$ and $Y$ are contractible, so $X$ is H-equivalent to $Y$, but one is unital while the other is not.
\end{exa}

\bigskip

\section{Irreducible digital H-spaces}\label{irr}
A digital image is said to be \emph{$\NP_i$-irreducible} if it is not $\NP_i$-homotopy equivalent to a digital image of fewer points. Every digital image is homotopy equivalent to an irreducible digital image, and the same is true for H-spaces and H-equivalence:

\begin{lem}\label{htpeqirreducible}
Any $\NP_i$-digital H-space is H-equivalent to an $\NP_i$-irreducible digital H-space for $i \in \{1,2\}$.
\end{lem}
\begin{proof}
Let $(X,e_X,\mu_X)$ be an $\NP_i$-digital H-space, and let $(Y,e_Y)$ be an $\NP_i$-irreducible digital image homotopy equivalent to $(X,e_X)$ by maps $f:(X,e_X) \to (Y,e_Y)$ and $g:(Y,e_Y) \to (X,e_X)$. Let $\mu_Y: Y\times Y \to Y$ be given by $$\mu_Y = f\circ \mu_X \circ (g\times g).$$ Then $(Y,e_Y,\mu_Y)$ is an $\NP_i$-digital H-space, and it is H-equivalent to $(X,e_X,\mu_X)$ by $f$ and $g$.
\end{proof}

In this section, we show some interesting properties of irreducible digital H-spaces.

The following is a slight variation of Lemma 2.8 of \cite{hmps}:
\begin{lem}\label{nonsurjection}
Let $X$ be a digital image. Then $X$ is $\NP_i$-irreducible if and only if no automorphism of $X$ is $\NP_i$-homotopic to a non-surjective map.
\end{lem}
\begin{proof}
In the contrapositive, the statement to be shown is that $X$ is reducible if and
only if some automorphism of $X$ is homotopic to a non-surjective map. Lemma 2.8 of \cite{hmps} shows that $X$ is reducible if and only if the identity function $\id_X : X \rightarrow X$ is homotopic to a non-surjective map (the entire setting of \cite{hmps} assumes the $\NP_1$ category, but the exact proof as written also works in the $\NP_2$ category). So it suffices to show that if $f$ is an automorphism homotopic to a non-surjection, then this implies that the identity function $\id_X : X \rightarrow X$ is homotopic to a non-surjection.

So we let $f:X\to X$ be an automorphism, and assume that $f\simeq_i g$, where $g$ is not a surjection. Then composing with $f^{-1}$ gives $$f\circ f^{-1} \simeq_i g\circ f^{-1},$$ and so $$\id_X \simeq_i g\circ f^{-1}$$ and $g\circ f^{-1}$ is not a surjection since $g$ is not a surjection.
\end{proof}

When $x\in X_e$, each multiplication map $\mu_x$ and $\nu_x$ is homotopic to the identity by Theorem \ref{homotopictoid}, and so the above gives:
\begin{cor}\label{muisomorphism}
If $X$ is $\NP_i$-irreducible and $x\in X_e$, then $\mu_x$ and $\nu_x$ are isomorphisms.
\end{cor}


Now we prove that, in a connected $\NP_i$-irreducible H-space, we must automatically have left and right inverses.

\begin{thm}\label{irreducibleinverses}
Let $(X,e,\mu)$ be a connected $\NP_i$-irreducible H-space. Then $(X,e,\mu)$ has left and right inverse maps $\alpha,\beta:X\to X$. If $X$ is homotopy-associative, then $\alpha \simeq_i \beta$. 
If $X$ is unital and associative, then $\alpha = \beta$.
\end{thm}
\begin{proof}
First we show that a right inverse exists. By Corollary \ref{muisomorphism}, the map $\mu_x:X\to X$ is an isomorphism. Let $x' = \mu_x^{-1}(e)$. 

Since $\mu_x$ is an isomorphism mapping $e$ to $x$, we will have $\mu_x(X_e) = X_e$, and so $x' = \mu_x^{-1}(e) \in X_e$. And we have: 
\[ \mu(x,x') = \mu_x(x') = \mu_x(\mu_x^{-1}(e)) = e, \]
and so $x'$ is a right-inverse element of $x$.

Now we show that in fact $x'$ is the only right-inverse element of $x$. Assume that we have some other element $x''$ with $\mu(x,x'')=e$. Then we have $$\mu_x(x') = \mu_x(x''),$$ and since $\mu_x$ is invertible this means $x'=x''$ and so the right-inverse element is unique.

Let $\beta: X \to X$ be the function sending each $x$ to its unique right inverse element. Since $$\mu(x,\beta(x))=e,$$ we have $$\mu \circ (\id_X,\beta) = c_e,$$ and so $\beta$ is a right inverse map, provided that it is continuous. To see that $\beta$ is continuous, take $a,b\in X$ with $a\sim b$, and we will show that $\beta(a)\sim \beta(b)$. We have:
\[ \mu(b,\beta(b)) = e = \mu(a,\beta(a)) \sim \mu(b,\beta(a)) \]
where the last adjacency is by continuity of $\mu$, since $a\sim b$. The above means that $$\mu_a(\beta(a) \sim \mu_b(\beta(a)).$$ Since $X$ is connected, $\mu_b$ is an isomorphism by Corollary \ref{muisomorphism}. Thus we have $\beta(b) \sim \beta(a)$ as desired.

So far we have shown that $X$ has a right inverse map $\beta$. Similar arguments show that it also has a left inverse map $\alpha$, and we have proven the first statement of the theorem.

Now assume that $X$ is homotopy-associative, and we will show that $\alpha \simeq_i \beta$. This follows from various homotopies already established. Since $X$ is an H-space, we have 
$$
\mu \circ (c_e, \id_X) \simeq_i \id_X \simeq_i \mu \circ (\id_X,c_e),
$$ and since $\alpha$ and $\beta$ are left- and right-inverses, we have 
$$
\mu\circ (\alpha,\id_X) = c_e = \mu\circ (\id_X,\beta).
$$ 
Combining these, together with homotopy-associativity, gives:
\[
\begin{split}
\alpha &\simeq_i \mu \circ (\id_X,c_e) \circ \alpha \\
&= \mu \circ (\alpha, c_e) \\
&\simeq_i \mu \circ (\alpha, \mu \circ (\id_X,\beta)) \\
&\simeq_i \mu \circ (\mu \circ (\alpha, \id_X), \beta) \\
&\simeq_i \mu \circ (c_e, \beta) \\
&= \mu \circ (c_e, \id_X) \circ \beta \\
& \simeq_i \beta
\end{split}
\]

If $X$ is unital and associative, all homotopies above become equalities, and so $\alpha = \beta$, proving the last statement of the theorem.
\end{proof}

The following is an improvement to Lemma \ref{htpeqirreducible}, showing that connected H-spaces are H-equivalent to irreducible left-unital H-spaces.
\begin{lem}\label{irreducibleleftunit}
Any connected $\NP_i$-digital H-space $(X,e,\mu)$ is homotopy equivalent as a digital H-space to a left-unital irreducible $\NP_i$-digital H-space $(X,p,\tau)$. 
If $\mu(e,x) = \mu(x,e)$ for all $X$, then $(X,p,\tau)$ is unital for $i \in \{1,2\}$.
\end{lem}
\begin{proof}
By Lemma \ref{htpeqirreducible}, we may assume without loss of generality that $(X,e,\mu)$ is $\NP_i$-irreducible.

For $x\in X$, let $\mu_x,\nu_x:X\to X$ denote left- and right-multiplications by $x$. By Corollary \ref{muisomorphism}, these maps are isomorphisms.

We define $p = \mu_e(e)$, and $\tau:X\to X$ by $$\tau = \mu\circ (\mu_e^{-1} \times \mu_e^{-1}).$$ To show $(X,e,\mu)$ and $(X,p,\tau)$ are H-equivalent, we will verify Definition \ref{hspaceequiv} using the maps $f=\mu_e:X \to X$ and $g=\mu_e^{-1}:X\to X$ for our homotopy equivalence. Clearly we have $$f\circ g = \id_X = g\circ f.$$ By Theorem \ref{homotopictoid} we have $$\mu_e \simeq_i \id_X \simeq_i \mu_e^{-1},$$ and also we have:
\[ 
\tau \circ (\mu_e \times \mu_e) = \mu \circ (\mu_e^{-1} \times \mu_e^{-1}) \circ (\mu_e \times \mu_e) = \mu \simeq_i \mu_e \circ \mu
\]
and 
\[ 
\mu_e^{-1} \circ \tau = \mu_e^{-1} \circ \mu\circ (\mu_e^{-1} \times \mu_e^{-1}) \simeq_i \mu \simeq_i \mu \circ (\mu_e^{-1}\times \mu_e^{-1})
\]
and thus $(X,e,\mu)$ and $(X,p,\tau)$ are H-equivalent.

Now to show $(X,\tau,p)$ is left-unital, take $a\in X$, and we have:
\[ \tau(p,a) = \mu(\mu_e^{-1}(p), \mu_e^{-1}(a)) = \mu(e,\mu_e^{-1}(a)) = \mu_e(\mu_e^{-1}(a)) = a \]
as desired.

For the second statement, assume additionally that $$\mu(e,x) = \mu(x,e)$$ for all $x\in X$, which is to say $\mu_e = \nu_e$. Then we have:
\[ \tau(a,p) = \mu(\mu_e^{-1}(a), \mu_e^{-1}(p)) = \mu(\nu_e^{-1}(a), e) = \nu_e(\nu_e^{-1}(a)) = a \]
and so $(X,p,\tau)$ is left-unital.
\end{proof}

We note that the result above still holds exactly as written using ``right-unital'' instead of ``left-unital''. So any H-space is H-equivalent to either a left-unital or a right-unital one, though we have not shown that we can achieve both simultaneously. 

Our classification of $\NP_2$-digital H-spaces in Theorem \ref{np2classification} will imply that any $\NP_2$-digital H-space is indeed H-equivalent to a unital H-space, but we are left with the following open question for the $\NP_1$ category:
\begin{quest}
Is every $\NP_1$-digital H-space H-equivalent to a unital H-space?
\end{quest}

\bigskip

\section{Any connected $\NP_2$-digital H-space is contractible}\label{cont}
A digital image is called \emph{$\NP_i$-rigid} when the identity is not $\NP_i$-homotopic to any other map. 
The paper \cite{hmps} presents examples of $\NP_1$-irreducible digital images which are not $\NP_1$-rigid. The standard example is the simple cycle of $n$ points with $n\ge 5$, in which a rotation map is $\NP_1$-homotopic to, but not equal to, the identity. In the $\NP_2$ category, this graph is $\NP_2$-irreducible, but it is also $\NP_2$-rigid because the rotation map is not $\NP_2$-homotopic to the identity. 

In fact we can show that there is no example of an irreducible non-rigid digital image in the $\NP_2$ category:
\begin{lem}\label{np2rigid}
Any $\NP_2$-irreducible digital image is $\NP_2$-rigid.
\end{lem}
\begin{proof}
Let $X$ be $\NP_2$-irreducible. To obtain a contradiction, assume that $X$ is not $\NP_2$-rigid. This means there is a map $f:X\to X$ different from the identity with $f\simeq_2 \id_X$. Theorem 3.2 of \cite{stae21} shows that any $\NP_2$-homotopy can be realized by a homotopy whose steps change only one point at a time. (The same result independently appears in \cite{chih}.) That is, we may assume that $f$ agrees with $\id_X$ at all points except one.

Let $c\in X$ be the single point for which $f(c)\neq c$, and $f(x)=x$ for all $x\neq c$. This means that $f(x)\neq c$ for all $x\in X$, and so $f$ is non-surjective. Then $\id_X \simeq_2 f$ means that the identity is $\NP_2$-homotopic to a non-surjection, and so by Lemma \ref{nonsurjection} our digital image $X$ is $\NP_2$-reducible as desired.
\end{proof}

The above lets us prove a very strong condition on any $\NP_2$-digital H-space.

\begin{thm}\label{np2contractible}
Let $(X,e,\mu)$ be an $\NP_2$-digital H-space. Then $X_e$, the connected component of $e$, is $\NP_2$-contractible.
\end{thm}
\begin{proof}
By Lemma \ref{htpeqirreducible}, we may assume that $X$ is $\NP_2$-irreducible, and thus by Lemma \ref{np2rigid} it is $\NP_2$-rigid. Thus the component $X_e$ is itself irreducible and rigid. By Theorem \ref{homotopictoid}, we have $\mu_x \simeq_2 \id_X$ for all $x\in X_e$, and thus by rigidity we have $\mu_x = \id_X$ for all $x \in X_e$.

We now have $$\mu \circ (\id_X,c_e)(x) = \mu(x,e) = e$$ for all $x\in X_e$, and thus $\mu\circ (\id_X,c_e)$ is constant on $X_e$. But since $(X,e,\mu)$ is an $\NP_2$-digital H-space, we have $$\mu\circ (\id_X,c_e)\simeq_2 \id_X.$$ Restricting to $X_e$, this implies that $\id_{X_e} \simeq_2 c_e$, which means that $X_e$ is $\NP_2$-contractible as desired.
\end{proof}

\begin{cor}
Any connected $\NP_2$-digital H-space is $\NP_2$-contractible.
\end{cor}

The results above are reminiscent of \cite{dtg}, which shows that any $\NP_2$-digital topological group consists only of contractible components (in fact, they must all be complete graphs). The same is not necessarily true of $\NP_2$-digital H-spaces: in Section \ref{examples} we give a construction of $\NP_2$-digital H-spaces in which the identity component is contractible, but the other components may not be.

\bigskip

\section{Pointed vs non-pointed homotopies} \label{pointed}
Given two pointed digital images $(X,a)$ and $(Y,b)$ and pointed maps $f,g:(X,a) \to (Y,b)$ (``pointed map'' means $f(a)=b$ and $g(a)=b$), we say $f$ and $g$ are ``pointed homotopic'' when they are homotopic by a homotopy in which $a$ maps to $b$ in every intermediate stage.

Given two pointed digital images $(X,a)$ and $(Y,b)$, we say they are \emph{$\NP_i$-pointed homotopy equivalent} when there are maps $f:X\to Y$ and $g:Y\to X$ with $g\circ f \simeq_i \id_X$ and $f\circ g\simeq_i \id_Y$ by pointed homotopies.

In our definition of H-space, we require that $$\mu\circ (c_e, \id_X)\simeq_i \id_X,$$ but we did not require that this homotopy be pointed. Say that an $\NP_i$-digital H-space $(X,e,\mu)$ is a \emph{pointed $\NP_i$-digital H-space} when these homotopies used in the definition are pointed homotopies. Following Definition \ref{hspaceequiv}, we say that two pointed $\NP_i$-digital H-spaces are \emph{pointed H-equivalent} when there exist homotopy equivalences as in Definition \ref{hspaceequiv} such that all of the homotopies in that definition are pointed.

In the classical theory of H-spaces, under mild topological assumptions (say, if the space is a CW-complex) any homotopy may be made pointed, and so the pointedness condition is not important. But in our digital setting the assumption of pointedness will make a difference in certain examples.

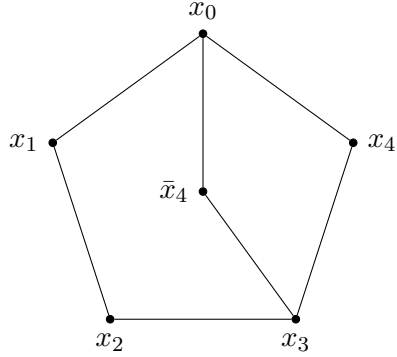
\begin{figure}[h!]
\begin{tikzpicture}[scale=2.1]
	\node[vertex] (y0)   at (0,1)  [label=above :$x_0$] {};
	\node[vertex] (y1)   at (162:1) [label=left :$x_1$] {}; 
	\node[vertex] (y2) at (234:1) [label=below :$x_2$] {};
	\node[vertex] (y3) at (306:1)  [label=below :$x_3$] {};
	\node[vertex] (y4) at (18:1) [label=right :$x_4$] {};
	\node[vertex] (y4p) at (0,0) [label=left :$\bar x_4$] {};
		
	\draw (y0) -- (y1) -- (y2) -- (y3) -- (y4) -- (y0);
	\draw (y3) -- (y4p) -- (y0);
\end{tikzpicture}
\vskip 0.1cm
\caption{\label{5twistfig} An H-space which is not a pointed H-space}
\end{figure}

First we give an example of an $\NP_1$-digital H-space which is not pointed. Let $X$ be the digital image in Figure \ref{5twistfig}, which was studied in \cite{hmps}.  This $X$ behaves in most cases like a simple cycle of 5 points, but one point has been ``replicated''. Let $\rho:X\to X$ be the function given by $\rho(\bar x_4)=x_4$, and $\rho(x)=x$ for all other points. Theorem 4.11 and Section 5 of \cite{hmps} show that $\rho$ is homotopic to $\id_X$, but not through any pointed homotopy. (The homotopy of $\id_X$ to $\rho$ is achieved in two stages by rotating all points by one position, and then rotating them back to their original position but collapsing the point originally at $\bar x_4$ into $x_4$. This homotopy cannot be modified to preserve the base point.)

\begin{exa}\label{unpointedex}
Let $X$ be the digital image of Figure \ref{5twistfig}. We can make $(X,x_0)$ into an $\NP_1$-digital H-space with multiplication $\mu$ given by:
\begin{itemize}
\item $\mu(x_i,x_j) = x_{i+j}$;
\item $\mu(x_i,\bar x_4) = \mu(\bar x_4, x_i) = x_{i+4}$; and
\item $\mu(\bar x_4,\bar x_4) = x_3$,
\end{itemize}
where all subscripts are read modulo 5. This $\mu$ is $\NP_1$ continuous.

To check the condition required for $(X,x_0,\mu)$ to be an (unpointed) H-space, we have:
\[ \mu \circ (\id_X, c_{x_0}) (x) = \mu(x,x_0) = \rho(x) \]
and so $$\mu \circ (\id_X, c_{x_0}) = \rho \simeq_1 \id_X,$$ and similarly $$\mu \circ (c_{x_0} , \id_X) =\rho \simeq_1 \id_X,$$ and so $(X,x_0,\mu)$ is a digital H-space.

But by the work in \cite{hmps}, these homotopies $\rho \simeq_1 \id_X$ cannot be made pointed, and so $(X,x_0,\mu)$ is not a pointed digital H-space.
\end{exa}

The same space, with a slightly different operation, gives a pointed H-space which is not pointed H-equivalent to a digital topological group.
\begin{exa}\label{nondtg}
Let $X$ be the digital image from Figure \ref{5twistfig} with basepoint $x_0$, but with multiplication $\tau:X\times X \to X$ given as follows:
\begin{itemize}
\item $\tau(x_i,x_j) = x_{i+j}$;
\item $\tau(x_0,\bar x_4) = \tau(\bar x_4,x_0) = \bar x_4$; 
\item $\tau(x_i,\bar x_4) = \tau(\bar x_4,x_i) = x_{i+4}$ for $i\neq 0$; and
\item $\tau(\bar x_4, \bar x_4) = x_3$.
\end{itemize}
This $\tau$ is the same as $\mu$ of Example \ref{unpointedex} except in how it handles $x_0$, which is a unit for $\tau$ but only a homotopy-unit for $\mu$. The operation $$\tau:X\times X \to X$$ is $\NP_1$-continuous, and since $x_0$ is a unit for $\tau$, we have $$\tau \circ (\id_X,c_{x_0}) = \id_X = \tau \circ (c_{x_0},\id_X)$$ and so $(X,x_0,\tau)$ is a pointed $\NP_1$-digital H-space.

We will show that $(X,x_0,\tau)$ is not pointed H-equivalent to any digital topological group. 
We say an $\NP_i$-digital H-space $(Y,e,\mu)$ is \emph{pointed homotopy associative} when $$\mu\circ (\id\times \mu) \simeq_i \mu \circ (\mu \times \id)$$ by some pointed $\NP_i$-homotopy $$H:(Y\times Y \times Y, (e,e,e)) \times [0,k] \to (Y,e).$$
Note that the proof of Theorem \ref{htpeqassociative} still holds in the pointed category; that is, if two pointed H-spaces are pointed H-equivalent and one is pointed homotopy associative, then the other is also pointed homotopy associative.

Since every digital topological group is associative, it is automatically pointed homotopy-associative. Thus to show that our $(X,x_0,\tau)$ is not pointed H-equivalent to an $\NP_1$-digital topological group, it will suffice to show that $(X,x_0,\tau)$ is not pointed homotopy-associative.

Let $d_1,  d_4:(X,x_0)\to (X,x_0)$ be the maps given as follows:
\[ d_1(a) = \begin{cases} x_0 &\text{ if } a=x_0, \\ x_1 &\text{ otherwise,}\end{cases} \qquad  d_4(a) = \begin{cases} x_0 &\text{ if } a=x_0, \\  x_4 &\text{ otherwise.}\end{cases} \]
For $a\in X$, we have:
\[ 
\begin{split}
\tau \circ (\tau \times \id_X) \circ (\id_X, d_1,  d_4) (a) &= \begin{cases} \tau(\tau(x_0, x_0), x_0) &\text{ if } a = x_0, \\ \tau(\tau(a,x_1),  x_4) &\text{ otherwise} \end{cases} \\
&= \begin{cases} x_0 &\text{ if } a = x_0, \\
\tau(x_{i+1},  x_4) &\text{ if } a = x_i \text{ for } i\neq 0, \\
\tau(x_0,  x_4) &\text{ if } a =  \bar x_4
\end{cases} \\
&= \begin{cases} x_0 &\text{ if } a = x_0, \\
x_i &\text{ if } a = x_i \text{ for } i\neq 0, \\
 x_4 &\text{ if } a = \bar x_4
\end{cases}
\\
&= \rho(a),
\end{split}
\]
and similarly
\[ 
\begin{split}
\tau \circ(\id_X \times \tau) \circ (\id_X, d_1, d_4)(a) &= 
\begin{cases} \tau(x_0, \tau(x_0, x_0)) &\text{ if } a = x_0, \\ \tau(a,\tau(x_1, x_4)) &\text{ otherwise} \end{cases} \\
&= \begin{cases} x_0 &\text{ if } a=x_0, \\
\tau(a, x_0) &\text{ otherwise} \end{cases} \\
&= a.
\end{split}
\]

The above formulas show that $$\tau \circ(\tau \times \id_X) \circ (\id_X, d_1, d_4) =\rho,$$ while $$\tau \circ (\id_X \times \tau) \circ (\id_X, d_1, d_4) = \id_X.$$ 

Since $\rho$ is not pointed homotopic to $\id_X$, the above means that $\tau\circ (\tau\times \id_X)$ is not pointed homotopic to $\tau \circ (\id_X\times \tau)$. Thus $(X,x_0,\tau)$ is not pointed homotopy associative, and so it is not pointed H-equivalent to any digital topological group.
\end{exa}

\bigskip

\section{Examples, and classification of $\NP_2$-digital H-spaces} \label{examples}
Obviously any $\NP_i$-digital topological group is an example of an $\NP_i$-digital H-space. By Theorem \ref{inducedmult}, any digital image homotopy equivalent to a digital topological group is an H-space. So it is natural to ask if there are any $\NP_i$-digital H-spaces which are not H-equivalent to an $\NP_i$-digital H-space. We give a class of examples in this section.

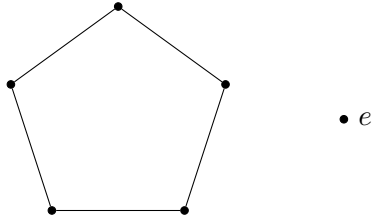
\begin{figure}[h!]
\[ 
\begin{tikzpicture}[scale=1.5]
	\node[vertex] (y0)   at (0,1) {};
	\node[vertex] (y1)   at (162:1) {}; 
	\node[vertex] (y2) at (234:1) {};
	\node[vertex] (y3) at (306:1) {};
	\node[vertex] (y4) at (18:1) {};
		
	\draw (y0) -- (y1) -- (y2) -- (y3) -- (y4) -- (y0);
	
	\node[vertex] at (2,0) [label=right:$e$] {};
\end{tikzpicture}		
\]
\caption{A digital H-space which is not homotopy equivalent to a digital topological group\label{disjointfig}.}
\end{figure}

An example is shown in Figure \ref{disjointfig}, which is the simplest possible case of the following construction. We say a pair $(Z,\tau)$ is an \emph{$\NP_i$-digital magma}  when $Z$ is a digital image, and $\tau:Z\times Z\to Z$ is $\NP_i$-continuous. (In a magma, we make no assumptions about existence of a homotopy-unit.) Such magmas are very easy to construct, even in the $\NP_2$ category. For example if $Z$ is any digital image then defining $\tau:Z\times Z \to Z$ as a constant gives an $\NP_2$-digital magma.

\begin{exa}\label{magmapoint}
Let $(Z,\tau)$ be any $\NP_i$-digital magma, and declare some new point $e\not \in Z$. Then we will show the disjoint union $X = Z\sqcup \{e\}$ is an $\NP_i$-digital H-space with unit element $e$. (Figure \ref{disjointfig} is such a union, where $Z$ is the 5-cycle graph, which is the digital topological group $\Z_5$.) 

Our H-space multiplication $\mu:X\times X \to X$ is defined as follows:
\[ 
\mu(a,b) = \begin{cases} \tau(a,b) \qquad &\text{ if }a,b\in Z, \\
a &\text{ if } b=e, \\
b &\text{ if } a=e
\end{cases}
\]

To show that $\mu$ is $\NP_i$-continuous, take $(a,b)\sim_i (c,d)$ in $X\times X$, and we consider various cases: If both of $a$ and $b$ are in $Z$, then both of $c,d$ are also in $Z$, so $$\tau(a,b) = \mu(a,b) \sim \mu(c,d) = \tau(c,d)$$ because $\mu$ is $\NP_i$-continuous. If $a \in Z$ and $b=e$, then also we must have $c\in Z$ and $d=e$, and we have $\tau(a,b) = a$ and $\tau(c,d) = c$ and so $\tau(a,b) \sim \tau(c,d)$ since $(a,b)\sim_i (c,d)$. Similarly if $a=e$ and $b\in Z$ we have $\tau(a,b) \sim \tau(c,d)$. 

The element $e$ is a unit, so $(X,e,\mu)$ is a digital H-space.

But $X$ is not, in general, H-equivalent to a digital topological group. Corollary 4.9 of \cite{dtg} states that the components of a digital topological group must be isomorphic to one another as graphs, but we have made no requirement that the components of $X$ be homotopy equivalent to one another. In particular, if $Z$ has any noncontractible component, then this component of $X$ is not homotopy equivalent to $\{e\}$. In this case, no digital image homotopy equivalent to $X$ can have all components isomorphic, and thus $X$ is not H-equivalent to any $\NP_i$-digital topological group.
\end{exa}

The following classification theorem shows that any irreducible $\NP_2$-digital H-space is generally of the type from Example \ref{magmapoint}: it consists of a singleton unit component $\{e\}$, together with an $\NP_2$-digital magma, and the H-space operation agrees with the magma operation except possibly mapping certain components of $X\times X$ to the identity element. 
\begin{thm}\label{np2classification}
Let $(X,e,\mu)$ be an irreducible $\NP_2$-digital H-space. Then $X$ is a disjoint union $X = \{e\} \sqcup Z$, and there exists a magma operation $\tau: Z \times Z \to Z$, and a subset $A\subseteq Z \times Z$ such that $A$ is a union of components and
\begin{equation}\label{magmanp2eq}
\mu(x,y) = \begin{cases}
\tau(x,y) &\text{ if } (x,y)\in A \\
x &\text{ if } y=e, \\
y &\text{ if } x=e, \\
e &\text{ otherwise.}
\end{cases} 
\end{equation}
\end{thm}
\begin{proof}
Since $X$ is irreducible, by Theorem \ref{np2contractible} the component of $e$ is $\{e\}$. By Theorem \ref{homotopictoid}, $\mu_e \simeq_2 \id_X$, but $X$ is rigid by Lemma \ref{np2rigid}, and so we must have $\mu_e = \id_X$. Similarly $\nu_e = \id_X$, and so $e$ is a unit, which means $\mu(x,e) = x$ and $\mu(e,y) = y$ which is part of what we must show.

Let $A \subset Z \times Z$ be defined by $A = \mu^{-1}(Z)$. To show $A$ is a union of components, it suffices to take $(a,b) \in A$ and $(c,d)\in Z$ with $(a,b)\sim_2 (c,d)$, and show that $(c,d)\in A$. Since $(a,b)\in A$ we have $\mu(a,b) \in Z$ so $\mu(a,b)\neq e$. By continuity of $\mu$, we have $\mu(c,d) \sim \mu(a,b) \neq e$, and since $\{e\}$ is a component this means that $\mu(c,d) \neq e$ and so $(c,d)\in A$ as desired.

Choose some $z\in Z$, and we define $\tau: Z \times Z \to Z$ by:
\[ 
\tau(x,y) = \begin{cases} \mu(x,y) &\text{ if } (x,y)\in A, \\
z &\text{ otherwise. }
\end{cases}
\]
Since $A$ is a union of components, the piecewise definition above makes $\tau:Z \times Z \to Z$ continuous.

Since  $A = \mu^{-1}(Z)$, we have $\mu(Z - A) = e$. Combining the above gives \eqref{magmanp2eq} as desired.
\end{proof}

The construction in Example \ref{magmapoint} will produce many examples of disconnected $\NP_i$-digital H-spaces which are not H-equivalent to $\NP_i$-digital topological groups. It is natural to ask if connected examples exist. In the $\NP_2$ category, they cannot: by Corollary \ref{np2contractible}, any connected $\NP_2$-digital H-space must be contractible, so it is H-equivalent to a single point, which is an $\NP_2$-digital topological group. In the $\NP_1$ category, we do not know if such examples can exist:

\begin{quest}\label{examplequestion}
Is there any connected $\NP_1$-digital H-space which is not H-equivalent to an $\NP_1$-digital topological group?
\end{quest}

\bigskip

\appendix

\section{Classification of $\NP_i$-digital topological groups} \label{dtg}

In this appendix, we resolve a question that was left unresolved in \cite{dtg}, and complete the full classification of $\NP_i$-digital topological groups. Section 6 of \cite{dtg} already classified the $\NP_2$ case, so we will focus on $\NP_1$.

Given a finitely generated group $G$ with some chosen finite subset $S\subseteq G$, the \emph{Cayley graph}  \cite[page 28]{BM} of $G$ with respect to $S$, which we denote $\Gamma(G,S)$, is the graph with vertex set $G$ and adjacency defined by $a\sim b$ if and only if $ab^{-1} \in S$ or $ba^{-1} \in S$.

The paper \cite{dtg} assumes that $S$ is a generating set for $G$, which is a common convention. Theorem 4.3 of \cite{dtg} shows that any Cayley graph is an $\NP_1$-digital topological group, but did not address the converse adequately because of an incorrect implicit assumption. 

Specifically, Example 4.6 of that paper presents an $\NP_1$-digital topological group and claims that it is not a Cayley graph, but this is not correct. Example 4.6 of \cite{dtg} is a complete graph of 4 points, with group operation making it isomorphic as a group to $\Z_4$. It was claimed in \cite{dtg} that this is not a Cayley graph, because the Cayley graph of $\Z_4$ is a simple cycle of 4 points, not a complete graph. But any given group $G$ may have several different Cayley graphs, depending on the choice of the set $S$. Since we do not require $S$ be a minimal generating set, in this example we may consider $G = \{0,1,2,3\}$ with generating set $S=\{1,2\}$, and then $\Gamma(G,S)$ is indeed a complete graph of 4 points.

In fact, if we do not require $S$ to be a generating set, then all $\NP_1$-digital topological groups are Cayley graphs:
\begin{thm}\label{cayleygraph}
Let $G$ be an $\NP_1$-digital topological group, and let $S\subset G$ be the set of all points adjacent to the group identity element. Then $G$ is the Cayley graph $\Gamma(G,S)$. (Note that $S$ may not be a generating set.)
\end{thm}
\begin{proof}
To show that $G$ is the stated Cayley graph, we will show that $a\sim b$ if and only if $ab^{-1}\in S$. Let $e\in G$ be the unit element.

Theorem 3.4 of \cite{dtg} shows that the right-multiplication $\nu_x$ is an isomorphism for any $x$. Thus we will have $a\sim b$ if and only if $$\nu_{b^{-1}}(a) \sim \nu_{b^{-1}}(b),$$ but this is equivalent to $ab^{-1} \sim e$, which means $ab^{-1}\in S$. 
\end{proof}

If the set $S$ is assumed to generate $G$, then the Cayley graph $\Gamma(G,S)$ will be connected. And we can show that any connected digital topological group arises in exactly this way:
\begin{thm}\label{connectedcayleygraph}
Let $G$ be a connected $\NP_1$-digital topological group, and as above let $S$ be the set of points adjacent to the identity element. Then $S$ is a generating set, and $G$ is the Cayley graph $\Gamma(G,S)$.
\end{thm}
\begin{proof}
The proof of Theorem \ref{cayleygraph} applies again to show that $G$ is the Cayley graph $\Gamma(G,S)$. It remains only to show that $S$ generates G. 

To obtain a contradiction, assume that $S$ does not generate $G$. Since $G$ is connected and $\langle S \rangle$, the set generated by $S$, is nonempty, this means that there is a pair of adjacent points $a\sim b$ with $a\in \langle S\rangle$ and $b\not\in \langle S\rangle$. 

Let $n$ be the order of $a$, so that $a^n=e$. Then since $a\sim b$, we have $$e = \nu_{a^{n-1}}(a) \sim \nu_{a^{n-1}}(b),$$ and thus $$ba^{n-1}\sim e,$$ which is to say that $ba^{n-1} \in S$. Since $a\in S$ this means $ba^{n-1} a \in \langle S\rangle$ and so $b\in \langle S\rangle$ which is a contradiction.
\end{proof}

The results above are a full classification of $\NP_1$-digital topological groups. We summarize these results together with the work in Section 6 of \cite{dtg}, which classified $\NP_2$-digital topological groups.
\begin{thm}
The class of $\NP_1$-digital topological groups is exactly the class of Cayley graphs $\Gamma(G,S)$. (Where $S$ may not generate $G$.) An $\NP_1$-digital topological group is connected if and only if $S$ generates $G$. 

The class of $\NP_2$-digital topological groups is exactly the class of finite cluster graphs (disjoint unions of complete graphs) having isomorphic clusters. An $\NP_2$-digital topological group is connected if and only if it is a complete graph.
\end{thm}

\bigskip

\bibliography{hspaces}
\bibliographystyle{plain}

\end{document}